\documentclass[11pt]{article} 

\usepackage[utf8]{inputenc} 

\usepackage{geometry} 
\usepackage{graphicx} 

\usepackage{booktabs} 
\usepackage{array} 
\usepackage{paralist} 
\usepackage{verbatim} 
\usepackage{subfig} 

\usepackage{fancyhdr} 
\usepackage[nottoc,notlof,notlot]{tocbibind} 
\usepackage[titles,subfigure]{tocloft} 

\usepackage{amsmath,amssymb,amsthm} 

\theoremstyle{plain} 
\newtheorem{lemma}[equation]{Lemma} 
\newtheorem{theorem}[equation]{Theorem} 
\newtheorem{remark}[equation]{Remark} 
\newtheorem{conjecture}[equation]{Conjecture} 

\theoremstyle{definition}
\newtheorem{definition}[equation]{Definition}

\numberwithin{equation}{section} 

\title{\bf LOWER BOUNDS FOR $L_1$ DISCREPANCY}
\author{Armen Vagharshakyan}
\date{} 

\begin{document}
\maketitle
\begin{abstract}
We find the best asymptotic lower bounds for the coefficient of the leading term of the $L_1$ norm of the two-dimensional (axis-parallel) discrepancy  that can be obtained by K.Roth's orthogonal function method among a large class of test functions.
We use methods of combinatorics, probability, complex and harmonic analysis.
\end{abstract}
\maketitle
\section{Introduction}\indent\indent
Let $P$ be a finite set in the two-dimensional unit square ($P\subset [0,1]^2$). 
Define the two-dimensional (axis-parallel) discrepancy function as follows:
\begin{equation*}
D(x,y)=D_P(x,y)=\sharp{(P\cap ([0,x]\times[0,y]))}-\sharp(P)\cdot xy,\quad (x,y)\in [0,1]^2,
\end{equation*}
where $\sharp(Q)$ denotes the number of elements (cardinality) of a set $Q$.
 The discrepancy function measures the difference between the actual number of points of the set $P$ in an axis-parallel rectangle $[0,x]\times [0,y]$ and the ``expected" number of points in that rectangle. Thus, the discrepancy function quantifies the ``closeness" of the distribution generated by a finite set $P$ to the uniform distribution.

The $L_\infty$ norm of the discrepancy function (also known as the star-discrepancy) -
$
||D_P||_{L_{\infty}([0,1]^2)}
$
 has been studied for a while and is known to be related to numerical integration, metric entropy, the small ball inequality etc.
We refer the reader to \cite{B},\cite{C} or \cite{JM} for a comprehensive discussion of the star-discrepancy.

In this paper we will be dealing with lower bounds  for the $L_1$ norm of the discrepancy function. Quoting a recent essay on the discrepancy theory (\cite{B}, 2013):
``The other endpoint of the $L^p$ scale, $p=1$, is not any less (and perhaps even more) difficult 
than the star-discrepancy estimates. The only information that is available
is the two-dimensional inequality (proved in the paper of Halasz \cite{H},1981))".

	When discussing the $L_1$ norm of the discrepancy function, it is convenient to define:
\begin{equation}\label{theconstant}
d_N=\frac{1}{\sqrt{\ln{N}}}\cdot\underset{\sharp{(P)}=N}{\inf}\int_0^1 \int_0^1 |D_P(x,y)| \; dx dy.
\end{equation}
 It can be deduced from H. Davenport's argument (\cite{D},1956) that:
\begin{equation}\label{upper}
\limsup_{N\to \infty} d_N<+\infty.
\end{equation}
On the other hand (see remark \ref{halasztest}) G. Halasz (\cite{H},1981) proved the following inequality:
\begin{equation}\label{lower}
\liminf_{N\to \infty} d_N\geq c_H>0.
\end{equation}
     A careful analysis of his proof shows that:
\begin{equation*}
c_H=\frac{1}{1152(\sqrt{e}+1)\sqrt{\ln 2}}\approx 0.00039.
\end{equation*}
The inequalities \eqref{upper} and \eqref{lower} show that $d_N$ is the coefficient of the leading term of the $L_1$ norm of the two-dimensional (axis-parallel) discrepancy function.

In this paper we improve the asymptotic lower bounds for $d_N$, namely:
\begin{theorem}\label{introtheorem}
For the coefficient of the leading term of the $L_1$ norm of the two-dimensional discrepancy defined by \eqref{theconstant} the following asymptotic lower bounds hold:
\begin{equation*}
\liminf_{N\to \infty}d_N\geq \frac{3}{256\sqrt{e\ln 2}}\approx 0.00854
\end{equation*}
and
\begin{equation*}
\limsup_{N\to \infty}d_N\geq \frac{1}{64\sqrt{e\ln 2}}\approx 0.01138.
\end{equation*}
\end{theorem}
For comparison, the best-known lower estimates of the $L_2$ norm of the discrepancy were obtained in \cite{HM}. 
\begin{remark}\label{comaprisonofnorms}
In \cite{FPS} it was proved that:
\begin{equation*}
\liminf_{N\rightarrow\infty}\frac{1}{\sqrt{\ln{N}}} \cdot\inf_{\sharp(P)=N}||D_P||_{L_2([0,1]^2)}\leq 0.17905.
 \end{equation*}
Numerical experiments for the sets introduced in \cite{BTY} make the following estimate plausible:
\begin{equation*}
 \liminf_{N\rightarrow\infty}\frac{1}{\sqrt{\ln{N}}}\cdot\inf_{\sharp(P)=N}||D_P||_{L_2([0,1]^2)}\leq 0.17601.
\end{equation*}
\end{remark}
\begin{remark}
Comparison of remark \ref{comaprisonofnorms} with theorem \ref{introtheorem} shows that for sets with minimal $L_2$ discrepancy their $L_2$ discrepancy is close to the $L_1$ discrepancy.
\end{remark}
In this paper we also show that theorem \ref{introtheorem} provides the best bounds that can be obtained by K. Roth's orthogonal function method for a large class of test functions (see theorems \ref{asasum} and \ref{naturalclass}). 

Theorem \ref{introtheorem}, as well as the fact that the best known lower bounds for the $\limsup$ and $\liminf$ of the coefficient of the leading term of the $L_2$ norm of the two-dimensional discrepancy are also different (this fact can be derived from \cite{HM}), lead us to formulate the following conjecture:
\begin{conjecture} The following strict inequality is true:
\begin{equation*}
\limsup_{N\rightarrow\infty} d_N>\liminf_{N\rightarrow\infty} d_N
\end{equation*}
\end{conjecture}

%
%
%
%
%
%
\section*{Acknowledgements}\indent\indent
I would like to thank Dmitriy Bilyk and Jill Pipher for their valuable comments.

%
%
%
%
%
%
\section{Auxiliary Functions}\label{orthogonality}\indent\indent
The existing proofs of lower bounds for discrepancy are based on certain construction of functions called Roth's auxiliary functions (see \cite{R}). 
In this section we will modify that construction so that now the auxiliary functions do not take the value zero, but  from all other perspectives they are as good as Roth's original auxiliary functions. 
The fact that they do not take the value zero will play a crucial role later on.

For a finite set $P\subset [0,1]^2$ denote by $N=\sharp(P)$ the number of elements of the set $P$. Let $n$ be the only integer satisfying the condition:
\begin{equation}\label{whatnis}
2^{n-1}<2\sharp(P)=2N\leq 2^n.
\end{equation}
For every $i=0,1,\dots ,n$ we will define 
 $E_i$ to be a family of dyadic rectangles in $[0,1]^2$:
\begin{equation*}
E_i=\bigcup_{l=0}^{+\infty}E_i^l,
\end{equation*}
where the sets $E_i^l$ will be defined recursively in the next paragraph, and we will define
 $C_i$ to be a family of dyadic rectangles in $[0,1]^2$:
\begin{equation*}
C_i=\bigcup_{l=0}^{+\infty}C_i^l,
\end{equation*}
where the sets $C_i^l$ will be defined recursively in the next paragraph.
\\\indent
\textbf{Recursive definition.}\\\indent
\textbf{Step $0$.}
In this step we define the families $E_i^0$ and $C_i^0$.
Consider the dyadic rectangles $R=R_x\times R_y\in [0,1]^2$ whose sidelengths satisfy
$|R_x|=2^{-i}$, $|R_y|=2^{i-n}$. Denote by $C_i^0$ the family of those sets $R$ that contain at least a point of the set $P$ and
by $E_{i}^0$ the family of those sets $R$ that do not contain any point of the set $P$.
We have:
\begin{equation*}
\sharp(C_{i}^0)+\sharp(E_{i}^0)=2^n
\end{equation*}
and
\begin{equation*}
\sharp(C_{i}^0)\leq N.
\end{equation*}\indent
\textbf{Step $1$.}
Partition each of the sets $R\in C_i^0$ into $2^{2(n+1)}$ dyadic rectangles which have the form
$R^1=R_x^1\times R_y^1$, where $|R_x^1|=|R_x|/2^{n+1}$, $|R_x^1|=|R_x|/2^{n+1}$. Define
$E_i^1$ to be the family of those sets $R^1$ which do not contain any point of the set $P$ and
define
$C_i^1$ to be the family of those sets $R^1$ which contain at least a point of the set $P$. 
We have:
\begin{equation*}
\sharp(C_i^1)\leq N.
\end{equation*}\indent
\textbf{Step $l+1$.}
In step $l+1$ partition each of the sets $R^{l}=R_x^l\times R_y^l\in C_i^{l}$ into $2^{2(n+1)}$ dyadic rectangles which have the form
$R^{l+1}=R_x^{l+1}\times R_y^{l+1}$, where 
\begin{equation*}
|R_x^{l+1}|=|R_x^l|/2^{n+1},\; |R_y^{l+1}|=|R_y^l|/2^{n+1}.
\end{equation*}
 Define
$E_i^{l+1}$ to be the family of those sets $R^{l+1}$ which do not contain any point of the set $P$ and
define
$C_i^{l+1}$ to be the family of those sets $R^1$ which contain at least a point of the set $P$. 
We have:
\begin{equation*}
\sharp(C_i^{l+1})\leq N.
\end{equation*}
\begin{definition}\label{haardefinition}
Let $R=R_x\times R_y$ be a dyadic rectangle in $[0,1]^2$. By $h_R$ we will denote the $L_{\infty}$ normalized Haar function adapted to the dyadic rectangle $R$, i.e.
\begin{equation*}
h_R(x,y)=h_{R_x}(x)\cdot h_{R_y}(y),
\end{equation*}
where the function $h_{R_x}$ equals $1$ on the left half of the dyadic interval $R_x$, $-1$ on the right half of the dyadic interval $R_x$ and $0$ outside of the dyadic interval $R_x$. Similarly, the function $h_{R_y}$ equals $1$ on the left half of the dyadic interval $R_y$, $-1$ on the right half of the dyadic interval $R_y$ and $0$ outside of the dyadic interval $R_y$.
\end{definition}
\begin{definition}\label{definitionauxiliary} Define the auxiliary functions $f_i$ by:
\begin{equation}
f_i=\sum_{R\in E_i} h_R,
\end{equation}
where $h_R$ is the $L_{\infty}$ normalized Haar function adapted to the rectangle $R$ (see definition \ref{haardefinition}). 
\end{definition}
\begin{remark}\label{originalauxiliary}
The following function:
\begin{equation*}
f_i^0=\sum_{R\in E_i^0} h_R
\end{equation*}
is the original Roth's auxiliary function.
\end{remark}
\begin{lemma}\label{setzerocalc}
The set  
\begin{equation*}
M=[0,1]^2\setminus \bigcup_{R\in E_i}R
\end{equation*}
has measure zero.
\end{lemma}
\begin{proof}
\begin{equation*}
\sum_{R\in C_i^l} |R|=\sharp(C_i^l)2^{-n-2(n+1)l}=N 2^{-n-2(n+1)l}\rightarrow 0, \text{ as } l\rightarrow +\infty.
\end{equation*}
\end{proof}
\begin{remark}\label{setzero}
Due to lemma \ref{setzerocalc} the function $f_i$ is defined almost everywhere on $[0,1]^2$. For definiteness, we will additionally take $f_i$ to be equal to $1$ on the remaining set of measure zero.
\end{remark}

We now investigate the properties of the functions $f_i$.
\begin{lemma}\label{squareone}
\begin{equation*}
f_i\colon [0,1]^2\rightarrow \lbrace -1,1 \rbrace
\end{equation*}
or equivalently
\begin{equation*}
|f_i(x,y)|\equiv 1.
\end{equation*}
\end{lemma}
\begin{proof}
The proof follows from the definition of the functions $f_i$ (see definition \ref{definitionauxiliary}) , from the remark \ref{setzero} and the fact that the elements of $E_i$ are disjoint.
\end{proof}
\begin{lemma}\label{short}
\begin{equation}
\int_0^1 \int_0^1 D_P f_i\leq -(2^n-N) \frac{N\cdot 2^{-2n}}{16}.
\end{equation} 
\end{lemma}
\begin{proof}
We use the following well-known fact about the interrelation between discrepancy and Haar functions (see e.g. \cite{C}): if the coordinate rectangle $R$ does not contain any point of the finite set $P$ then:
\begin{equation*}
\int_0^1\int_0^1 D_P h_R=-\frac{\sharp(P) |R|^2}{16}.
\end{equation*}
Let $E_i^0$ be as in the construction of the function $f_i$, then we have:
\begin{equation*}
\int_0^1\int_0^1 D_P\cdot f_i\leq\sum_{R\in E_i^0} \int_0^1\int_0^1 D_P\cdot h_R=
-\sum_{R\in E_i^0} \frac{N |R|^2}{16}\leq -(2^n-N) \frac{N\cdot 2^{-2n}}{16}.
\end{equation*}
\end{proof}
\begin{lemma}\label{long}
Let $0\leq i_1<i_2<\dots <i_p\leq n$.
Then
\begin{equation*}
a)\quad\int_0^1\int_0^1 f_{i_1}f_{i_2}\dots f_{i_p}=0,
\end{equation*}
and
\begin{equation*}
b)\quad\left|\int_0^1\int_0^1 D_P\cdot f_{i_1}f_{i_2}\dots f_{i_p}\right|\leq 2^{i_1-i_p}\frac{N2^{-n}}{16}.
\end{equation*}
\end{lemma}
\begin{proof}
Let 
\begin{equation*}
R_{i_1}\in E_{i_1},R_{i_2}\in E_{i_2},\dots, R_{i_p}\in E_{i_p}.
\end{equation*}
Denote:
\begin{equation*}
R_{i_1}=R_{i_1,x}\times R_{i_1,y},R_{i_2}=R_{i_2,x}\times R_{i_2,y},\dots, R_{i_p}=R_{i_p,x}\times R_{i_p,y}
\end{equation*}
For definiteness assume:
\begin{equation*}
|R_{i_1,x}|=2^{-i_1-(n+1)l_1},|R_{i_1,y}|=2^{i_1-n-(n+1)l_1},
\end{equation*}
\begin{equation*}
|R_{i_2,x}|=2^{-i_2-(n+1)l_2},|R_{i_2,y}|=2^{i_2-n-(n+1)l_2},
\end{equation*}
\begin{equation*}
\vdots
\end{equation*}
\begin{equation*}
|R_{i_p,x}|=2^{-i_p-(n+1)l_p}, |R_{i_p,y}|=2^{i_p-n-(n+1)l_p}.
\end{equation*}
Note that the repeated subdivisions into $2^{2(n+1)}$ dyadic rectangles in the construction  in section \ref{orthogonality} guarantees that
 the lengths of the dyadic intervals $R_{i_1,x},\dots, R_{i_p,x}$ are mutually distinct,
similarly the lengths of the dyadic intervals  $R_{i_1,y},\dots, R_{i_p,y}$ are mutually distinct. Hence the product:
\begin{equation*}
h_{R_1}h_{R_2}\dots h_{R_p}
\end{equation*}
is either identically zero or (up to a sign) the Haar function on the intersection:
\begin{equation*}
R=\bigcap_{t=1}^p R_{i_t}.
\end{equation*}
Thus the first claim of the theorem follows.
\\
We now estimate the sidelengths of the dyadic rectangle $R$. Write $R=R_x\times R_y$ then:
\begin{equation*}
|R_x|\leq 2^{-i_p}, |R_y|\leq 2^{i_1-n},
\end{equation*}
and therefore
\begin{equation}\label{rectanglesize}
 |R|=|R_x\times R_y|\leq 2^{-i_p+i_1-n}.
\end{equation}
Now denote:
\begin{equation*}
E=\left\lbrace R=\bigcap_{t=1}^p R_{i_t}\;\colon \; R_{i_1}\in E_{i_1}, \dots, R_{i_p}\in E_{i_p}\right\rbrace.
\end{equation*}
Note that the elements of $E$ are mutually disjoint (as the elements of each $E_{i_t}$ are mutually disjoint). We also have:
\begin{equation*}
f_{i_1}f_{i_2}\dots f_{i_p}=\sum_{R\in E}\pm h_R.
\end{equation*}
Using \eqref{rectanglesize} we get:
\begin{equation*}
\left|\int_0^1\int_0^1 D_P\cdot f_{i_1}f_{i_2}\dots f_{i_p}\right|\leq\sum_{R\in E}\left|\int_0^1\int_0^1 D_P\cdot h_R\right|=
\end{equation*}
\begin{equation*}
=\sum_{R\in E}\frac{N|R|^2}{16}\leq \frac{N2^{-i_p+i_1-n}}{16} \sum_{R\in E} |R|=2^{i_1-i_p}\frac{N2^{-n}}{16}.
\end{equation*}
\end{proof}
\section{Combinatorial properties of the auxiliary functions}
We now explore the combinatorial properties of the auxiliary functions $f_i$ defined by \ref{definitionauxiliary}. We adapt the following notation from the theory of symmetric functions (see e.g. \cite{M}):
\begin{equation}\label{productsoffis}
e_1=e_1^n=\sum_{i=0}^n f_i,
\end{equation}
\begin{equation*}
e_2=e_2^n=\sum_{0\leq i_1<i_2 \leq n} f_{i_1}f_{i_2},
\end{equation*}
\begin{equation*}
e_3=e_3^n=\sum_{0\leq i_1<i_2<i_3\leq n} f_{i_1}f_{i_2}f_{i_3},
\end{equation*}
and so on.\\\indent
Let the number $k$ be odd ($k=1,3,5,\dots$). Then using the fact that $f_i^2\equiv 1$ (see lemma \ref{squareone}) we can write:
\begin{equation}\label{definitionofais}
(f_0+f_1+\dots+f_n)^k=A_1^n(k)e_1^n+A_3^n(k) e_3^n+A_5^n(k)e_5^n+\dots
\end{equation}
(here we assume that $A_p^n(k)=0$ once $p>\min(k,n+1)$).
\begin{remark}
For example we compute:
\begin{equation*}
(f_0+f_1)^3=f_0^3+3f_0f_1^2+3f_0^2f_1+f_1^3=4f_0+4f_1,
\end{equation*}
so that
$
A_1^1(3)=4.
$
\end{remark}
\begin{remark}\label{meaningofcomb}
It is important to point out that the coefficient of the term of \begin{equation*}(f_0+f_1+\dots+f_n)^k\end{equation*}  that is linear with respect to 
$
f_0+f_1+\dots+f_n$
is
$
A_1^n(k).
$
\end{remark}

\section{Test Functions}\label{symmetric}\indent\indent
We start this section by sketching Roth's orthogonal function method (as modified by Halasz \cite{H}).

Let $P\subset [0,1]^2$ be a finite set and let $f_0,f_1,\dots,f_n$ be the corresponding functions defined by \ref{definitionauxiliary}.
Let $H_n=H_n(z_0,z_1,\dots,z_n)$ be an odd, symmetric function that is entire in each of its $n+1$ variables. 

By $LIN(H_n)$ we will denote the coefficient of the term of the composite function $H_n(f_0,f_1,\dots,f_n)$ which is linear with respect to
$
\sum_{i=0}^n f_i.
$
For example, according to the remark \ref{meaningofcomb}:
$
LIN((f_0+\dots+f_n)^k)=A_1^n(k).
$

Assume that
\begin{equation}\label{normalizationassumption}
||H_n(z_0,z_1,\dots,z_n)||_{L_\infty([-1,1]^{n+1})}=1.
\end{equation}

Then we can estimate the $L_1$ norm of the discrepancy function $D_P(x,y)$ from below as follows:
\begin{equation*}
||D_P||_1\geq   \left|\int_0^1\int_0^1 D_P\cdot  H_n(f_0,f_1,\dots,f_n) \right|\geq LIN(H_n)\left| \sum_{i=0}^n\int_0^1\int_0^1 D_Pf_i\right|+Error,
\end{equation*}
where in Roth's orthogonal function method it is further proved that the term $Error$ is negligible
(estimating the term $Error$ from above by means of lemma \ref{long}).
Then lemma \ref{short} is applied to obtain the following lower bound for the $L_1$ norm of the discrepancy function:
\begin{equation}\label{witherror}
||D_P||_1\geq LIN(H_n)(n+1)(2^n-N)\frac{N2^{-2n}}{16}+Error,
\end{equation}
where the integers $N$ and $n$ are defined by \eqref{whatnis}.
\\\indent
Informally, Roth's orthogonal function method (as modified by Halasz \cite{H}) consists of finding appropriate functions $H_n$ (called test functions)
for which $LIN(H_n)$ grows as fast as possible, whereas the error term $Error$ is negligible, so that an appropriate lower estimate for $||D_P||_1$ follows.
\begin{remark}\label{halasztest}
In his proof G. Halasz \cite{H} used the following test function:
\begin{equation*}
G=\prod_{i=0}^n \left(1+\frac{i\gamma}{\sqrt{\log N}}f_i^0\right)-1,
\end{equation*}
where the auxiliary functions $f_i^0$ are as in remark \ref{originalauxiliary} and $\gamma>0$ is a certain small constant. This test function involves complex numbers. The choice of the test function $G$ was motivated by Riesz products. We will use a different approach by choosing our test function to be a solution of an extremal problem in Fourier analysis (see theorem \ref{naturalclass}).
\end{remark}
\begin{remark}\label{organization} The rest of the paper is organized as follows. We will first show that the problem of finding appropriate test functions $H_n$ reduces to finding a certain function $T$ (see
theorem \ref{asasum}), so that the quantity corresponding to $LIN(H_n)$ is:
\begin{equation}\label{linearpart}
LIN_n(T)=LIN\left(T\left(\frac{f_0+\dots+f_n}{\sqrt{n}}\right)\right).
\end{equation}
Then we will derive a precise formula for $LIN_n(T)$ (see section \ref{extreme}). After that we
will find the function $T$ that maximizes the quantity:
\begin{equation}
\limsup \sqrt{n}\cdot LIN_n(T)
\end{equation}
in a certain natural class of functions (see theorem \ref{naturalclass}) and show that for that function the error
term $Error$ appearing in \eqref{witherror} can be neglected (see section \ref{main}).
\end{remark}
\begin{theorem}\label{asasum}
Without loss of generality one can assume that the test functions $H_n$ have the following form:
\begin{equation*}
H_n(f_0,f_1,\dots,f_n)=T\left(\frac{f_0+f_1+\dots+f_n}{\sqrt{n}}\right).
\end{equation*}
\end{theorem}
\begin{proof}
According to the "Main theorem on symmetric functions" (see e.g. \cite{M}) the symmetric function
$H_n(f_0,f_1,\dots,f_n)$ can be written as a function depending only on the functions $e_i$ (introduced by \eqref{productsoffis}). We borrow the following notation from the theory of symmetric functions:
\begin{equation*}
p_k=\sum_{0\leq j\leq n}f_{j}^k,\quad 1\leq k\leq n+1.
\end{equation*}
Recalling Newton's identities:
\begin{equation*}
ke_k=\sum_{i=1}^k (-1)^{i-1} e_{k-i} p_i,
\end{equation*}
we obtain that a symmetric function can be written as an analytic function depending only on the functions $p_i$. Taking into account that $f_i^2=1$ (see lemma \ref{squareone}) we have:
\begin{equation*}
p_k=p_1,\quad 1\leq k\leq n+1,\quad k \text{ is odd},
\end{equation*}
and
\begin{equation*}
p_k=n+1=p_0,\quad 1\leq k\leq n+1,\quad k\text{ is even}.
\end{equation*}
Thus, the function $H_n$ depends only on $p_0$ and $p_1$, where $p_0$ is a constant. Recalling that $H_n$ is odd, it turns out that $H_n$ has to depend only on:
\begin{equation*}
p_1=f_0+\dots+f_n,
\end{equation*} i.e. we can write the function $H_n$ in the following form:
\begin{equation*}
H_n(f_0,f_1,\dots,f_n)=F_n(f_0+\dots+f_n).
\end{equation*}
Introduce functions $T_n$ as follows:
\begin{equation*}
F_n(f_0+\dots+f_n)=T_n\left(\frac{f_0+\dots+f_n}{\sqrt{n}}\right).
\end{equation*}
Then the condition \eqref{normalizationassumption} becomes:
\begin{equation*}
\sup_{(-\sqrt{n},\sqrt{n})}|T_n(x)|=1.
\end{equation*}

Recall that $\lbrace T_n (z)\rbrace_{n=1}^{\infty}$ is a family of odd, entire functions. Either it is normal or it goes locally uniformly to $0$, or to $\infty$. The latter two cases are not interesting, so let us assume that $T_n$ has a subsequence that tends to a certain odd, entire function $T$ locally uniformly in the complex plane. For clarity of notation we will assume that $T_n$ itself tends to $T$.

Note that the functions $f_i$ are non-correlated on the unit square and take the values $1$ and $-1$ with equal probability, thus:
\begin{equation*}
\frac{f_0+\dots+f_n}{\sqrt{n}}\rightarrow N(0,1).
\end{equation*}
So that the conditions: \begin{equation*}
\sup_{(-\sqrt{n},\sqrt{n})}|T_n(x)|=1
\end{equation*} imposed on the functions $T_n$ induce the following condition on the limiting function $T$ :
\begin{equation*}
\sup_{-\infty<x<+\infty}|T(x)|=1,
\end{equation*}
and instead of discussing the functions:
\begin{equation*}
T_n \left(\frac{f_0+\dots+f_n}{\sqrt{n}}    \right)
\end{equation*}
we can discuss the functions:
\begin{equation*}
T \left(\frac{f_0+\dots+f_n}{\sqrt{n}}    \right).
\end{equation*}
\end{proof}
%
%
%
%
%
%

\section{Generating function of $A_1^n(k)$}\indent\indent
In this section we derive the formula \eqref{meaningofprob} for the generating function of the sequence $\lbrace A_1^n(k)\rbrace_{k=1,3,\dots}$ introduced by the identity \eqref{definitionofais}.

Denote by $B_n$ the binomial distribution, i.e. the distribution of the sum of $n$ independent random variables which take
values $1$ and $0$ with probability $1/2$. 
Denote by $S_n$ the distribution of $n$ independent random variables which take values $1$ and $-1$ with probability $1/2$.
Then
\begin{equation*}
S_n=2B_n-n.
\end{equation*}
Denote by $M_{B_n}$ the moment-generating function of the binomial distribution $B_n$, that is:
\begin{equation*}
M_{B_n}(t)=E(e^{t B_n})=\sum_{k=0}^{\infty}\frac{1}{k!}E\left(\left(B_n\right)^k\right)t^k,
\end{equation*}
where $E(X)$ is the mathematical expectation of a random variable with distribution $X$. 
\newline
We know that:
\begin{equation*}
M_{B_n}(t)=\left(\frac{1}{2}+\frac{1}{2}e^t\right)^n.
\end{equation*}
As a consequence we get:
\begin{equation*}
M_{2B_{n}-n+1}(t)=E\left(e^{(2B_{n}-n+1)t}\right)=e^{(-n+1)t}M_{B_{n}}(2t)=
\end{equation*}
\begin{equation*}
=e^{(-n+1)t}\left(\frac{1}{2}+\frac{1}{2}e^{2t}\right)^{n}=e^{t}\left(\frac{e^t+e^{-t}}{2}\right)^{n}.
\end{equation*}

Let $x_0,x_1,\dots,x_n$ be independent random variables, where $x_0\equiv 1$ and $x_1,\dots , x_n$  take
values $1$ and $-1$ with probability $1/2$. Let $k$ be an odd number. We have:
\begin{equation*}
E\left((x_0+\dots+x_n)^k\right)=A_1^n(k).
\end{equation*}
But on the other hand, the sum $x_0+\dots+x_n$ has the distribution $1+S_{n}$. We note that $1+S_{n}=2B_{n}-n+1$, so that:
\begin{equation}\label{meaningofprob}
\sum_{k=1,3,\dots}\frac{1}{k!}A_1^n(k)t^k=\sum_{k=1,3,\dots}\frac{1}{k!}E\left((x_0+x_1+\dots+x_n)^k\right)t^k=
\end{equation}
\begin{equation*}
=\frac{
M_{2B_n-n+1}(t)-M_{2B_n-n+1}(-t)
}
{2}=\frac{e^t-e^{-t}}{2}\left(
\frac{e^t+e^{-t}}{2}
\right)^n.
\end{equation*}

\section{Precise formula for the linear part}\label{extreme}

Recall (see equation \eqref{linearpart}) that:
\begin{equation*}
LIN_n(T)=LIN\left(T\left(\frac{f_0+\dots+f_n}{\sqrt{n}}\right)\right).
\end{equation*}
In this section we will derive the precise formula \eqref{conciseformula3} for $LIN_n(T)$.\\
\noindent
We have:
\begin{equation*}
T\left(\frac{f_0+\dots+f_n}{\sqrt{n}}   \right)=T^{\prime}(0)\frac{\sum_{i=0}^n f_i}{\sqrt{n}}+\frac{T^{(3)}(0)}{3!}\frac{\left(\sum_{i=0}^n f_i\right)^3}{(\sqrt{n})^3}+\dots.
\end{equation*}
The coefficient of the term of this expansion that depends on $\sum_{i=0}^n f_i$ linearly is:
\begin{equation*}
LIN_n(T)= \left[ \frac{T^{\prime}(0)}{\sqrt{n}}+  
 \frac{T^{(3)}(0)}{3!(\sqrt{n})^3}A_1^{n}(3)+\frac{T^{(5)}(0)}{3!(\sqrt{n})^3}A_1^{n}(5)+\dots
 \right],
\end{equation*}
or in terms of the differentiation operator $\partial$:
\begin{equation*}
LIN_n(T)=
 \left[ \frac{\partial}{\sqrt{n}}+  
 \frac{\partial^3}{3!(\sqrt{n})^3}A_1^{n}(3)+\frac{\partial^5}{3!(\sqrt{n})^3}A_1^{n}(5)+\dots
 \right](T)(0).
\end{equation*}
\\\noindent
Using the formula \eqref{meaningofprob}, we get: 
\begin{equation}\label{conciseformula2}
LIN_n(e^{i\omega x})=
\left(\frac{e^{i\omega/\sqrt{n}}-e^{-i\omega/\sqrt{n}}}{2}\right)\left(
\frac{e^{i\omega/\sqrt{n}}+e^{-i\omega/\sqrt{n}}}{2}
\right)^n,
\end{equation}
\\\noindent
or in general:
\begin{equation}\label{conciseformula3}
LIN_n(T)=
\left[\left(\frac{e^{\partial/\sqrt{n}}-e^{-\partial/\sqrt{n}}}{2}\right)\left(
\frac{e^{\partial/\sqrt{n}}+e^{-\partial/\sqrt{n}}}{2}
\right)^n       \right](T)(0).
\end{equation}


%
%
%
%
%
%
%
%

\section{Solution of the extremal problem}\label{solution}\indent\indent
Recall (see remark \ref{organization}) that we are looking for an odd function $T$ (introduced in theorem \ref{asasum}) satisfying:
\begin{equation*}
\sup_{-\infty<x<\infty}|T(x)|=1,
\end{equation*}
for which the quantity:
\begin{equation*}
\limsup \sqrt{n} \cdot LIN_n(T)
\end{equation*}
is maximal. 

The following theorem proves that under some additional conditions on $T$ this extremal problem has a solution and that solution is:
\begin{equation*}
T(x)=\sin(x).
\end{equation*}
\begin{theorem}\label{naturalclass} If the function $T$ satisfies the following conditions:
\\
a) the support of the Fourier transform of $T$ is countable, and the sequence
$\lbrace \omega_j \rbrace_{j=1}^{\infty}$  is linearly independent over the set $Z$ of all integers:
\begin{equation}\label{whatist}
T(x)=\sum_{j=1}^{\infty} c_j e^{-i\omega_j x},
\end{equation}
\\
b) $T$ is an odd function,
\\
c) $||T||_{\infty}=1$
\\
then
\begin{equation}\label{theinequality}
\limsup \sqrt{n} \cdot LIN_n(T)\leq \frac{1}{\sqrt{e}}
\end{equation}
and the maximum is obtained for the function:
\begin{equation*}
T(x)=\sin(x).
\end{equation*}
\end{theorem}
\begin{remark}\label{formalderivative}
In theorem \ref{naturalclass} by $LIN_n(T)$ we understand the operator $LIN_n$ formally applied to the series \eqref{whatist} by formula \eqref{conciseformula2}.
\end{remark}
\begin{remark}
In theorem \ref{naturalclass} the inequality \eqref{theinequality} is understood in the sense that if the left hand side is real then the inequality is satisfied.
\end{remark}
\begin{proof}
Write
\begin{equation*}
T(x)=\sum_{j=1}^{\infty} c_j e^{-ix\omega_j}.
\end{equation*}
Note that according to \cite{K} (page 183, paragraph 6.9, theorem 9.3) we have:
\begin{equation*}
\sup_{-\infty<x<\infty}|T(x)|=\sum_{j=1}^{\infty}|c_j|.
\end{equation*}
Hence:
\begin{equation*}
\sum_{j=1}^{\infty}|c_j|=1.
\end{equation*}
On the other hand, according to the remark \ref{formalderivative} we get:
\begin{equation*}
\limsup \sqrt{n} \cdot LIN_n(T)=\lim  \sqrt{n} \cdot LIN_n(T)=\left[\partial e^{\partial^2/2} \right](T)(0).
\end{equation*}
We calculate:
\begin{equation*}
\left[\partial e^{\partial^2/2}\right](T)(x)=-i\sum_{j=1}^{\infty} \omega_j e^{-\omega_j^2/2}e^{-ix\omega_j}\cdot c_j.
\end{equation*}
Plugging in $x=0$, we get:
\begin{equation*}
\left[\partial e^{\partial^2/2}\right](T)(0)=-i\sum_{j=1}^{\infty} \omega_j e^{-\omega_j^2/2}\cdot c_j.
\end{equation*}
In order to continue with the proof note the following remark:
\begin{remark}
We have:
\begin{equation*}
\max_{-\infty<\omega<\infty} \omega e^{-\omega^2/2}=\frac{1}{\sqrt{e}},
\end{equation*}
and the maximum is obtained at the point $\omega=1$.
Similarly:
\begin{equation*}
\min_{-\infty<\omega<\infty} \omega e^{-\omega^2/2}=-\frac{1}{\sqrt{e}},
\end{equation*}
and the minimum is obtained at the point $\omega=-1$.
\end{remark}
\noindent Thus, in order that the maximum of the quantity:
\begin{equation*}
-i\sum_{j=1}^{\infty} \omega_j e^{-\omega_j^2/2}\cdot c_j
\end{equation*}
subject to the restriction:
\begin{equation*}
\sum_{j=1}^{\infty} |c_j|=1
\end{equation*}
is obtained we need the sequence $\lbrace \omega_j \rbrace$ to consist of only $2$ points, namely: $1$ and $-1$.
As a consequence we get that the odd function $T$ that maximizes the quantity:
\begin{equation*}
\left[\partial e^{\partial^2/2}\right](T)(0)
\end{equation*} 
subject to the conditions stated in the theorem is:
\begin{equation*}
T(x)=\frac{i}{2}e^{-ix}+\frac{(-i)}{2}e^{ix}=\sin(x),
\end{equation*}
and in that case:
\begin{equation*}
\left[\partial e^{\partial^2/2}\right](t)(0)=\frac{1}{\sqrt{e}}.
\end{equation*}
\end{proof}
%
%
%
%
%
%
%
%
\section{Derivation of the best constant}\label{main}
\begin{theorem}
We have
\begin{equation*}
\liminf_{N\rightarrow \infty} \frac{1}{\sqrt{\ln{N}}} \inf_{\sharp(P)=N} \int_0^1\int_0^1 |D_P|\geq \frac{3}{256\sqrt{e\ln{2}}}
\end{equation*}
and
\begin{equation*}
\limsup_{N\rightarrow \infty} \frac{1}{\sqrt{\ln{N}}} \inf_{\sharp(P)=N} \int_0^1\int_0^1 |D_P|\geq \frac{1}{64\sqrt{e\ln{2}}}
\end{equation*}
\end{theorem}
\begin{proof}
Taking $T(x)=\sin(x)$ we get:
\begin{equation*}
T\left(\frac{f_0+\dots+f_n}{\sqrt{n}}\right)=\sin\left(\frac{f_0+\dots+f_n}{\sqrt{n}}\right)=Im\left[\prod_{j=0}^n\left(\cos\left(\frac{1}{\sqrt{n}}\right)+i\cdot f_j\cdot \sin\left(\frac{1}{\sqrt{n}}\right)\right)\right]=
\end{equation*}
\begin{equation*}
=\cos^n\left(\frac{1}{\sqrt{n}}\right)\sin\left(\frac{1}{\sqrt{n}}\right)\left(\sum_{j=0}^n f_j\right)
-3! \cos^{n-2}\left(\frac{1}{\sqrt{n}}\right)\sin^{3}\left(\frac{1}{\sqrt{n}}\right)\left(\sum_{0\leq j_1<j_2<j_3} f_{j_1}f_{j_2}f_{j_3}\right)+\dots.
\end{equation*}
So that:
\begin{equation*}
\int_0^1\int_0^1 D_P \sin\left(\frac{f_0+\dots+f_n}{\sqrt{n}}\right) =
\end{equation*}
\begin{equation*}
=\cos^n\left(\frac{1}{\sqrt{n}}\right)\sin\left(\frac{1}{\sqrt{n}}\right)\left(\sum_{j=0}^n\int_0^1\int_0^1 D_P f_j \right)-
\end{equation*}
\begin{equation*}
-3! \cos^{n-2}\left(\frac{1}{\sqrt{n}}\right)\sin^{3}\left(\frac{1}{\sqrt{n}}\right)\left(\sum_{0\leq j_1<j_2<j_3} \int_0^1\int_0^1 D_P f_{j_1}f_{j_2}f_{j_3}\right)+\dots.
\end{equation*}
For the first term using lemma \ref{short} we estimate:
\begin{equation*}
\cos^n\left(\frac{1}{\sqrt{n}}\right)\sin\left(\frac{1}{\sqrt{n}}\right)\left(\sum_{j=0}^n\int_0^1\int_0^1 D_P f_j \right)\geq
\end{equation*}
\begin{equation*}
\geq \cos^n\left(\frac{1}{\sqrt{n}}\right)\cdot \left(\frac{1}{\sqrt{n}}+o(1)\right)(n+1) (2^n-N)\frac{N2^{-2n}}{16}\geq \left(\frac{1}{\sqrt{e}}+o(1)\right)\sqrt{n}(2^n-N)\frac{N2^{-2n}}{16}.
\end{equation*}
Let us point out that the appearance of:
\begin{equation*}
\frac{1}{\sqrt{e}}
\end{equation*}
here is not surprising as it corresponds to $LIN_n(\sin)$ (see section \ref{solution}).
\\
For the second term we apply lemma \ref{long} to get:
\begin{equation*}
3! \cos^{n-2}\left(\frac{1}{\sqrt{n}}\right)\sin^{3}\left(\frac{1}{\sqrt{n}}\right)\sum_{0\leq j_1<j_2<j_3}\left| \int_0^1\int_0^1 D_P f_{j_1}f_{j_2}f_{j_3}\right|\leq
\end{equation*}
\begin{equation*}
\leq 3! \frac{1}{(\sqrt{n})^3}\frac{N2^{-n}}{16}\sum_{0\leq j_1<j_2<j_3\leq n}  2^{j_1-j_3}
\leq 3! \frac{1}{(\sqrt{n})^3}\frac{N2^{-n}}{16}\cdot \sum_{d=1}^n 2^{-d} d(n-d)
=O\left(\frac{1}{\sqrt{n}}\right),
\end{equation*}
and the higher order terms are estimated similarly.
\\
Hence we can write:
\begin{equation*}
\inf_{2^{n-1}<2\sharp(P)=2N\leq 2^n}||D_P||_1\geq \inf_{2^{n-1}<2\sharp(P)=2N\leq 2^n}\left| D_P H_n\right|\geq
\end{equation*}
\begin{equation*}
\geq \left(\frac{1}{\sqrt{e}}+o(1)\right) \sqrt{n}(2^n-N)\frac{N2^{-2n}}{16}+O\left(\frac{1}{\sqrt{n}}\right).
\end{equation*}
Thus since:
\begin{equation*}
2^{n-1}<2\sharp(P)=2N\leq 2^n
\end{equation*}
then
\begin{equation*}
\liminf d_N\geq \frac{3}{256\sqrt{e\ln{(2)}}}
\end{equation*}
and
\begin{equation*}
\limsup d_N\geq \frac{1}{64\sqrt{e\ln{(2)}}}
\end{equation*}
in fact:
\begin{equation*}
d_{2^{n-1}}\geq \frac{1}{\sqrt{\ln(2^{n-1})}}\frac{1}{64\sqrt{e}}\sqrt{n}\rightarrow \frac{1}{64\sqrt{e\ln(2)}}.
\end{equation*}
\end{proof}

\end{document}